\documentclass[11pt]{article}
\usepackage{amscd, amsmath, amssymb, amsthm, hyperref}

\title{Symmetric differentials and the fundamental group}
\author{Yohan Brunebarbe, Bruno Klingler, and Burt Totaro}
\date{  }

\def\Z{\text{\bf Z}}
\def\Q{\text{\bf Q}}
\def\R{\text{\bf R}}
\def\C{\text{\bf C}}
\def\P{\text{\bf P}}
\def\F{\text{\bf F}}
\def\arrow{\rightarrow}

\def\X{\widetilde{X}}

\def\End{\text{End}}
\def\Re{\text{\, Re\, }}
\def\rank{\text{rank}}
\def\FF{\overline{\F_p}}

\begin{document}
\maketitle
\newtheorem{theorem}{Theorem}[section]
\newtheorem{corollary}[theorem]{Corollary}
\newtheorem{lemma}[theorem]{Lemma}

\theoremstyle{definition}
\newtheorem{definition}[theorem]{Definition}
\newtheorem{example}[theorem]{Example}

\theoremstyle{remark}
\newtheorem{remark}[theorem]{Remark}

H\'el\`ene Esnault asked whether a smooth complex projective variety
$X$ with infinite fundamental group must have a nonzero
symmetric differential, meaning that $H^0(X,S^i\Omega^1_X)\neq 0$
for some $i>0$. This was prompted by the second author's work
\cite{Klingler}. In fact, Severi had wondered in 1949 
about possible relations between
symmetric differentials and the fundamental group
\cite[p.~36]{Severi}.
We know from Hodge theory that the cotangent bundle $\Omega^1_X$
has a nonzero section
if and only if the abelianization of $\pi_1X$ is infinite.
The geometric meaning of other symmetric differentials is
more mysterious, and it is intriguing that they may have such
a direct relation to the fundamental group.

In this paper we prove the following result
on Esnault's question, in the slightly broader setting
of compact K\"ahler manifolds.

\begin{theorem}
\label{main}
Let $X$ be a compact K\"ahler manifold. Suppose that
there is a finite-dimensional representation of $\pi_1X$ over some field
with infinite image. Then $X$ has a nonzero symmetric differential.
\end{theorem}

All known varieties with infinite fundamental
group have a finite-dimensional complex representation with infinite image,
and so the theorem applies to them. Depending on what we know
about the representation, the proof
gives more precise lower bounds on the ring
of symmetric differentials.

\begin{remark}
(1) One reason to be interested in symmetric differentials
is that they have implications toward Kobayashi
hyperbolicity. At one extreme, if $\Omega^1_X$ is ample,
then $X$ is Kobayashi hyperbolic \cite[Theorem 3.6.21]{Kobayashispaces}.
(Equivalently, every holomorphic map $\C\arrow X$ is constant.)
If $X$ is a surface
of general type with $c_1^2>c_2$, then Bogomolov showed that
$\Omega^1_X$ is big and deduced that $X$ contains only finitely
many rational or elliptic curves, something which remains open
for arbitrary surfaces of general type \cite{Bogomolov, Deschamps}.
For any $\alpha\in H^0(X,S^i\Omega^1_X)$ with $i>0$,
the restriction of $\alpha$ to any rational curve in $X$ must be zero
(because $\Omega^1_{\P^1}$ is a line bundle of negative degree),
and so any symmetric differential
gives a first-order algebraic differential equation
satisfied by all rational curves in $X$.

A lot is already known about Kobayashi hyperbolicity in the situation
of Theorem \ref{main}. In particular, Yamanoi showed that for any smooth
complex projective variety $X$ such that $\pi_1X$
has a finite-dimensional complex representation whose image
is not virtually abelian,
the Zariski closure of any holomorphic map
$\C\arrow X$ is a proper subset of $X$ \cite{Yamanoi}.

(2) Arapura used Simpson's theory of representations
of the fundamental group to show that if $\pi_1X$ has a non-rigid
complex representation, then $X$ has a nonzero symmetric differential
\cite[Proposition 2.4]{Arapura}, which we state as Theorem
\ref{nonrigid}.

Thus the difficulty for Theorem \ref{main}
is how to use a rigid representation of the fundamental
group. The heart of the proof 
is a strengthening of Griffiths
and Zuo's results on variations of Hodge structure \cite{Griffithscurv,
Zuoneg}, from
weak positivity of the cotangent bundle (analogous to ``pseudoeffectivity''
in the case of line bundles) to bigness. As a result, we get many
symmetric differentials on the base of a variation of Hodge
structure.

The strengthening involves two ingredients. First, a new curvature
calculation shows that if a compact Kahler manifold $X$
has nonpositive holomorphic bisectional curvature,
and the holomorphic sectional curvature is negative
at one point, then the cotangent bundle of $X$ is nef and big
(Theorem \ref{curv}). Next, we have to relate any variation
of Hodge structure to one with discrete monodromy group.
For that, we use results of Katzarkov and Zuo which
analyze $p$-adic representations of the fundamental group
by harmonic map techniques. This reduction to the case
of discrete monodromy group is very much in the spirit
of the arguments by Eyssidieux, Katzarkov, and others
which prove the Shafarevich conjecture for linear groups
\cite{Eyssidieux, EKPR, Katzarkov}.

(3) The abundance conjecture in minimal model theory
would imply that a smooth complex projective variety $X$ is
rationally connected if and only if $H^0(X,(\Omega^1_X)^{\otimes i})=0$
for all $i>0$ \cite[Corollary 1.7]{GHS}.
Without abundance, Campana used Gromov's
$L^2$ arguments \cite{Gromov} on the universal cover to show
that if $X$ is not simply connected, then
$H^0(X,(\Omega^1_X)^{\otimes i})\neq 0$ for some $i>0$
\cite[Corollary 5.1]{Campana}.
But this conclusion is weaker than that of Theorem \ref{main}.
In particular,
finding a section of a general tensor bundle $(\Omega^1_X)^{\otimes i}$
has no direct implication towards Kobayashi hyperbolicity.
(There are more subtle implications, however.
Demailly has shown that every smooth projective variety $X$
of general type has some algebraic differential
equations, typically not first-order,
which are satisfied by all holomorphic maps
$\C\arrow X$ \cite[Theorem 0.5]{DemaillyMorse}.)

There are many varieties $X$ of general type (which have many sections
of the line bundles $K_X^{\otimes j}$, hence of the bundles
$(\Omega^1_X)^{\otimes i}$) which have no symmetric differentials.
For example, Schneider showed that a smooth subvariety $X\subset \P^N$
with $\dim(X)> N/2$ has no symmetric differentials \cite{Schneider}. Most
such varieties are of general type.

(4) The possible implication from infinite fundamental
group to existence of symmetric differentials cannot be reversed.
In fact, Bogomolov constructed smooth complex projective
varieties which are simply connected but have ample cotangent bundle
\cite[Proposition 26]{Debarreample}. Brotbek recently gave a simpler
example of a simply connected variety with ample cotangent bundle:
a general complete intersection surface of high multidegree in $\P^N$
for $N\geq 4$ has ample cotangent bundle \cite[Corollary 4.8]{Brotbek}.
The ring of symmetric differentials on such a variety is as big
as possible, roughly speaking. 

(5) Theorem \ref{main} makes it natural to ask whether the fundamental
group of a smooth
complex projective variety $X$, if infinite,
must have a finite-dimensional complex representation
with infinite image. This is not known. We know by Toledo that the
fundamental group of a smooth projective variety need not be
residually finite \cite{Toledo}, in particular need not be linear.

Even if the fundamental group of a smooth projective variety $X$
is infinite and residually finite, it is not known whether $\pi_1X$
always has a finite-dimensional
complex representation with infinite image. Indeed, there are
infinite, residually finite, finitely presented groups $\Gamma$ such that
every finite-dimensional complex representation of $\Gamma$ has finite
image. Such a group can be constructed as follows; can it be the
fundamental group of a smooth complex projective variety? Let $K$ be a global
field of prime characteristic $p$ (for example $K = \F_p(T)$).
Let $S$ be a finite set of primes
of $K$ and $O_S$ the subring of $S$-integers of $K$. Then
$\Gamma:= SL(n,O_S)$ is finitely
presented if $n\geq 3$ and $|S| > 1$ \cite{Splitthoff},
and any finite-dimensional complex representation of $\Gamma$
has finite image if $n\geq 3$ and $|S| > 0$
\cite[Theorem 3.8(c)]{Margulis}. Also, $\Gamma$ is residually finite.

(6) One cannot strengthen Theorem \ref{main} to say that every
non-simply-connected variety has a nonzero symmetric differential.
For example, Kobayashi showed that every
smooth complex projective variety $X$
with torsion first Chern class
and finite fundamental group has $H^0(X,S^i\Omega^1_X)=0$
for all $i>0$ \cite{Kobayashitensor}. This applies to Enriques surfaces,
which have fundamental group $\Z/2$.
\end{remark}

Acknowledgements: It is a pleasure to thank H.~Esnault,
who suggested the question 
leading to Theorem \ref{main} after a lecture by the second author on
\cite{Klingler}; F.~Campana for explaining \cite{Campana};
and I.~Dolgachev, C.~Haesemeyer, L.~Katzarkov,
and K.~Zuo for useful questions.

\vspace{0.2cm}
Convention: Throughout the paper, varieties and
manifolds are understood to be connected.

\section{Negatively curved varieties}
\label{curvsect}

\begin{theorem}
\label{curv}
Let $X$ be a compact K\"ahler manifold
with nonpositive holomorphic bisectional curvature.
Suppose that the holomorphic sectional curvature is negative
at one point of $X$.
Then the cotangent bundle of $X$ is nef and big.
\end{theorem}

For a vector bundle $E$ on a compact complex manifold,
write $P(E)$ for the projective bundle of hyperplanes in $E$.
We define a vector bundle $E$ to be {\it ample, nef}, or {\it big }if
the line bundle $O(1)$ on $P(E)$ has the corresponding property
\cite[Definition 6.1.1, Example 6.1.23]{Lazarsfeld}; see
Demailly \cite[Definition 6.3]{DemaillyL2} for the definition
of a nef line bundle on a compact complex manifold. It follows that
$E$ is big if and only if there are $c>0$ and $j_0\geq 0$ such that
$$h^0(X,S^jE)\geq cj^{\dim(X)+\rank(E)-1}$$
for all $j\geq j_0$. Note that Viehweg and Zuo use ``big'' for
a stronger property of vector bundles, as discussed below
in Remark (2).

We give the definition of holomorphic bisectional
curvature in the proof of Lemma \ref{curvbound}. On a K\"ahler
manifold, the holomorphic bisectional curvature $B(x,y)$ for
tangent vectors $x$ and $y$ is a positive linear combination
of the (Riemannian) sectional curvatures of the real 2-planes
$\R\{x,y\}$ and $\R\{x,iy\}$
\cite[section 1]{GK}. Holomorphic sectional curvature is a special
case of holomorphic bisectional curvature; it is also equal
to the sectional curvature of a complex line in the tangent space,
viewed as a real 2-plane. It follows that a K\"ahler manifold
with negative or nonpositive
sectional curvature satisfies the corresponding inequality for holomorphic
bisectional curvature, and that in turn implies the corresponding
inequality for holomorphic sectional curvature.

For example, Theorem \ref{curv} applies to the quotient of any
bounded symmetric domain by a torsion-free cocompact lattice,
or to any smooth subvariety of such a quotient. (This uses
that holomorphic bisectional curvature and holomorphic
sectional curvature decrease on complex submanifolds
\cite[section 4]{GK}.) Thus
we have a large class of smooth projective varieties
with a lot of symmetric differentials. Theorem \ref{curv} seems
to be new even for these heavily studied varieties. For a quotient
$X$ of a symmetric domain of ``tube type'', it was known
that, after passing to some finite covering,
$S^n\Omega^1_X$ contains the ample line bundle $K_X=\Omega^n_X$,
and so $X$ has some symmetric differentials \cite[section 4.2]{CDS}. But
that argument does not show that $\Omega^1_X$ is big.

\begin{remark}
(1) If the holomorphic {\it bisectional }curvature of 
a compact K\"ahler manifold is {\it negative},
then the cotangent bundle is ample. But
the cotangent bundle of $X$ need not be ample under
the assumptions of Theorem \ref{curv}, even if the holomorphic
sectional curvature is everywhere negative.
A simple
example is the product $C_1\times C_2$ of two curves of genus at least 2,
for which the natural product metric has negative holomorphic sectional
curvature and nonpositive holomorphic bisectional curvature.
The cotangent bundle is $\pi_1^*\Omega^1_{C_1}\oplus \pi_2^*\Omega^1_{C_2}$,
which is not ample because its restriction to a curve $C_1\times p$
has a trivial summand. A more striking example is the quotient $X$
of the product of two copies of the unit disc
by an irreducible torsion-free cocompact lattice. (Some surfaces
of this type are known as quaternionic
Shimura surfaces.)
The curvature conditions of Theorem \ref{curv}
are again satisfied at every point. In this case,
Shepherd-Barron showed that the algebra of symmetric differentials
$\oplus_{i\geq 0} H^0(X,S^i\Omega^1_X)$ is not finitely generated
\cite{SB}. It follows that the cotangent bundle of $X$ is nef and big
but not semi-ample \cite[Example 2.1.30]{Lazarsfeld}, let alone ample.

(2) The cotangent bundle need not be big in Viehweg's stronger sense
under the assumptions of Theorem \ref{curv}. To give the definition,
let $X$ be a projective variety over a field with an ample line bundle $L$.
A vector bundle $E$ is {\it weakly positive }if there is a nonempty
open subset $U$ such that for every $a>0$ there is a $b>0$ such
that the sections of $S^{ab}(E)\otimes L^{\otimes b}$ over $X$ span that
bundle over $U$. A bundle $E$ is {\it Viehweg big }if there is
a $c>0$ such that $S^c(E)\otimes L^{-1}$ is weakly
positive. The assumptions of Theorem \ref{curv} do not imply
that $\Omega^1_X$ is Viehweg big, as shown again by $X$ the product
of two curves of genus at least 2 \cite[Example 1.8]{Jabbusch}.

(3) Following Sakai, we define the {\it cotangent dimension }$\lambda(X)$
of a compact complex $n$-fold $X$ to be the smallest number $\lambda$
such that there is a positive constant $C$ with 
$\sum_{i=0}^jh^0(X,S^i\Omega^1_X)\leq Cj^{\lambda+n}$ for all $j\geq 0$
\cite{Sakai}. Then
$\lambda(X)$ is an integer between $-n$ and $n$, and
$\Omega^1_X$ is big if and only if $\lambda(X)$ has the maximum value, $n$.
\end{remark}

\begin{proof} (Theorem \ref{curv})
Let $\P(\Omega^1_X)\arrow X$ be the bundle of hyperplanes in
the cotangent bundle $\Omega^1_X$.
Since $X$ has nonpositive bisectional curvature, the associated
metric on the line bundle
$O(1)$ on $\P(\Omega^1_X)$ has nonnegative curvature
\cite[2.36]{Griffithspos}. It follows that $\Omega^1_X$ is nef.

The hard part is to show that $\Omega^1_X$ is big. Equivalently,
we have to show that the line bundle $O(1)$ on $\P(\Omega^1_X)$ is big.
By Siu, this holds if the differential form $(c_1O(1))^{2n-1}$,
which we know is nonnegative,
is positive at some point of the compact complex
manifold $\P(\Omega^1_X)$ \cite{Siu}.
The pushforward of the cohomology
class $(c_1O(1))^{2n-1}$ to $X$ is the Segr\'e class $s_n(TX)$.
In fact, this is true at the level of differential forms,
by Guler \cite{Guler}.
(The total Segr\'e class of a vector bundle $E$ is defined as the inverse
of the total Chern class, $s(E)=c(E)^{-1}$. For example, $s_1(E)=-c_1(E)$
and $s_2(E)=(c_1^2-c_2)(E)$.)
So we want to show that the Segr\'e number $\int_X s_n(TX)$
is positive (rather than zero).

\begin{lemma}
\label{segre}
Let $E$ be a holomorphic hermitian vector bundle of rank $n$ on a complex
manifold $X$ of dimension at least $n$. Let $p$ be a point in $X$.
Suppose that the curvature $\Theta_E$ in
$A^{1,1}(\End(E))$ is nonnegative at $p$, and that there
is a nonzero vector $e$ in $E_p$ such that the $(1,1)$-form
$\Theta_E(e,e)$ at $p$ is positive. Then the Segr\'e form
$s_n(E^*)$ is positive at $p$.
\end{lemma}

\begin{proof}
Since $E$ has nonnegative curvature, the associated
metric on the line bundle
$O(1)$ on $\P(E)$ has nonnegative curvature form
$c_1O(1)$. The Segr\'e form $s_n(E^*)$
is the integral along the fibers of the differential form
$(c_1O(1))^{2n-1}$. So the Segr\'e form is positive at
the point $p$ if the $(1,1)$-form $c_1O(1)$ is positive
at least at one point $e$ of the fiber over $p$. 
Griffiths's formula
for the curvature of $O(1)$ \cite[2.36]{Griffithspos} is:
$$c_1O(1)(y,y)=\frac{\Theta_E(e,e,y_h,y_h)}{|e|^2} + \omega(y_v,y_v),$$
where $y$ is a tangent vector in $\P(E)$ with horizontal
and vertical parts $y_h$ and $y_v$, and $\omega$ is a positive
$(1,1)$ form on the projective space $\P(E_p)$.
So the form $c_1O(1)$ is positive at a point
in $\P(E)$ if the corresponding vector $e$ in $E_p^*$ (defined
up to $\C^*$)
satisfies $\Theta_E(e,e,v,v)>0$ for all $v\neq 0$
in $T_pX$. This is exactly our assumption.
\end{proof}

The theorem is now a consequence of the following geometric lemma.
\end{proof}

\begin{lemma}
\label{curvbound}
Let $X$ be a K\"ahler manifold. Suppose that
at a point $p$, the holomorphic sectional curvature of $X$
is at most a negative constant
$-A$.
Then there is a nonzero vector
$x$ in $T_pX$ such that the holomorphic bisectional curvature
$B(x,y)$ is at most $-A/2$ for all nonzero vectors
$y$ in $T_pX$.
\end{lemma}

It seems surprising that the lemma holds without assuming
that $X$ has nonpositive holomorphic bisectional curvature.
(That is true in our application, however.) The bound $-A/2$
is optimal, as shown by the invariant metric on the complex unit $n$-ball
for $n\geq 2$: if we scale the metric to have holomorphic
sectional curvature equal to $-A$, then the holomorphic bisectional
curvature $B(x,y)$ varies between $-A$ (when $x$ and $y$ span
the same complex line) and $-A/2$ (when the hermitian inner
product $\langle x,y\rangle$ is zero).

In the special case where $X$ is a bounded symmetric domain,
Mok studied in detail
the geometry of the tangent vectors $x$ such that the holomorphic
bisectional curvature $B(x,y)$ is zero
for some $y$ \cite[p.~100 and p.~252]{Mok}. He called such vectors
``higher characteristic vectors''.

\begin{proof}
The curvature of a holomorphic vector bundle $E$ with hermitian
metric over a complex manifold $X$ can be viewed as a form
$$\Theta_E: E_p\times E_p\times T_pX\times T_pX\arrow \C$$
which is linear in the first and third variables and conjugate
linear in the second and fourth variables \cite[section 7.5]{Zheng}.
It satisfies
$$\Theta_E(y,x,w,z)=\overline{\Theta_E(x,y,z,w)}.$$
Taking the dual bundle changes the sign of the curvature,
in the sense that
$$\Theta_{E}(x,y,z,w) = -\Theta_{E^*}(y^*, x^*, z, w),$$
where $x^*$ and $y^*$ are the dual vectors given by the hermitian form.

We define the curvature $\Theta=\Theta_{TX}$ 
of a hermitian metric on $X$ to be
the curvature of the tangent bundle
as a holomorphic vector bundle with hermitian
metric.
When the metric is K\"ahler, we also have
$$\Theta(x,y,z,w)=\Theta(z,y,x,w).$$

The holomorphic bisectional curvature is defined by
$$B(x,y)=\frac{\Theta(x,x,y,y)}{|x|^2|y|^2}$$
for nonzero vectors $x,y\in T_pX$. By the identities above, the
holomorphic bisectional curvature is real and depends only on the complex
lines $\C x$ and $\C y$. The holomorphic sectional curvature
is
$$H(x)=B(x,x)=\frac{\Theta(x,x,x,x)}{|x|^4}$$
for a nonzero vector $x$ in $T_pX$. This depends only
on the complex line $\C x$.

Let $x$ be a nonzero vector in $T_pX$ which
maximizes the holomorphic sectional curvature. This is possible,
because the holomorphic sectional curvature is $C^{\infty}$
on the complex projective space of lines in $T_pX$. Write
$H(x)=-A<0$. 
With this simple choice, we will show that $B(x,y)\leq -A/2$ for
all nonzero vectors $y$ in $T_pX$. 

We can scale $x$ and $y$ to have length 1.
To first order, for $c\in \C$ near 0, we have
\begin{align*}
H(x+cy)&=\frac{1}{|x+cy|^4}\Theta(x+cy,x+cy,x+cy,x+cy)\\
&= (1-4\Re(\overline{c}\langle x,y\rangle))[H(x)
+4\Re(\overline{c}\Theta(x,x,x,y))] +
O(|c|^2)\\
&= H(x)+4\Re[\overline{c}(-H(x)\langle x,y\rangle +\Theta(x,x,x,y))] +
O(|c|^2),
\end{align*}
using that $|x+cy|^2=1+2\Re(\overline{c}\langle x,y\rangle)+|c|^2$.
(We take the hermitian metric $\langle x,y\rangle$ on $T_pX$
to be linear in $x$ and conjugate linear in $y$.)
Since the holomorphic sectional curvature is maximized at the vector $x$,
the first-order term in $c$ must be zero for all $c\in \C$, and so
$\Theta(x,x,x,y)=H(x)\langle x,y\rangle$.

Next, we compute to second order, for $c\in \C$ near 0.
The identities on curvature imply that $B(x,y)=\Theta(x,x,y,y)
=\Theta(y,x,x,y)=\Theta(y,y,x,x)=\Theta(x,y,y,x)$, and we know
that $\Theta(x,x,x,y)=H(x)\langle x,y\rangle$. Therefore:
\begin{align*}
H(x+cy)&=\frac{1}{|x+cy|^4}\Theta(x+cy,x+cy,x+cy,x+cy)\\
&= [1-4\Re(\overline{c}\langle x,y\rangle)-2|c|^2+12\Re(\overline{c}\langle
x,y\rangle)^2]\\
& \cdot [H(x)+4H(x)\Re(\overline{c}\langle x,y\rangle)+4B(x,y)|c|^2
+2\Re(\overline{c}^2\Theta(x,y,x,y))] + O(|c|^3)\\
&= H(x)-2H(x)|c|^2-4H(x)(\Re \overline{c}\langle x,y\rangle)^2
+4B(x,y)|c|^2+2\Re(\overline{c}^2\Theta(x,y,x,y))+O(|c|^3).
\end{align*}
Since the holomorphic sectional curvature is maximized at $x$,
the quadratic term in $c$ must be $\leq 0$ for all $c\in \C$. 
The term $-4H(x)(\Re \overline{c}\langle x,y\rangle)^2$ is nonnegative,
which works to our advantage. Let $c$ belong to one of the real lines in $\C$
such that $\Re(\overline{c}^2\Theta(x,y,x,y))=0$;
then we must have
$$-2H(x)|c|^2+4B(x,y)|c|^2\leq 0.$$
Therefore, $B(x,y)\leq H(x)/2=-A/2$, as we want.
\end{proof}

\section{More varieties with big cotangent bundle}

\begin{corollary}
\label{biratbig}
Let $X$ be a compact K\"ahler manifold
with nonpositive holomorphic bisectional curvature.
Let $Y$ be a compact K\"ahler manifold with a generically
finite meromorphic map $Y\dashrightarrow X$. Suppose that
the holomorphic sectional curvature of $X$ is negative
at some point in the closure of the image of $Y$.
Then the cotangent bundle of $Y$ is big.
\end{corollary}

\begin{proof}
By resolution of singularities, there is a compact K\"ahler manifold
$Y_2$ with a bimeromorphic morphism $Y_2\arrow Y$
such that the given map $f:Y\dashrightarrow X$ extends
to a morphism $f:Y_2\arrow X$. Since the ring of symmetric differentials
on a compact complex manifold is a bimeromorphic invariant, it suffices
to show that $\Omega^1_{Y_2}$ is big.

Let $W\arrow X$ be the Grassmannian bundle
of subspaces of $TX$ of dimension equal to $n=\dim(Y)$. Since
$f$ is generically finite, 
the derivative of $f$ is injective on an open dense subset $U$ of $Y_2$.
So the derivative
of $f$ gives a meromorphic map $g:Y_2\dashrightarrow W$ lifting $f$,
a morphism over $U$.
Again, there is a compact K\"ahler manifold $Y_3$ with a bimeromorphic
morphism $Y_3\arrow Y_2$, an isomorphism over $U$,
such that $g$ extends to a morphism
$g:Y_3\arrow W$. It suffices to show that $\Omega^1_{Y_3}$ is big.

There is a natural vector bundle $E$ of rank $n$
on the Grassmannian bundle $W$, a quotient
of the pullback of $\Omega^1_X$ to $W$. The bundle $E$
inherits a hermitian metric from $\Omega^1_X$.
Therefore the bundle $g^*E$ on $Y_3$ has a hermitian metric. Moreover,
the restriction of $g^*E$ to $U\subset Y_3$ can be identified 
with $\Omega^1_U$, with the metric pulled back from the metric on
$\Omega^1_X$ via the immersion $f:U\arrow X$.
Since $X$ has nonpositive bisectional curvature, and bisectional
curvature decreases on complex submanifolds
\cite[section 7.5]{Zheng},
the curvature of
$g^*E$ is nonnegative over $U$, hence over all of $X_3$.
Also, by Lemmas \ref{segre} and \ref{curvbound}, the Segr\'e form
$s_n((g^*E)^*)$ is positive at some point of $U$, because
$X$ has negative holomorphic sectional curvature at some point
in the image of $U$, and holomorphic sectional
curvature decreases on complex submanifolds \cite[section 7.5]{Zheng}.
The Segr\'e form may not be positive
on all of $Y_3$, but positivity on $U$ implies
that the number $\int_X s_n((g^*E)^*)$ is positive. Equivalently, the
line bundle $O(1)$ on $\P(g^*E)\arrow Y_3$ has nonnegative curvature
and the number $(c_1O(1))^{2n-1}$ is positive. So $O(1)$ is nef and big
on $\P(g^*E)$. Equivalently, $g^*E$ is nef and big on $Y_3$.

Because we can pull back 1-forms, we have a natural map
$\alpha:f^*\Omega^1_X\arrow \Omega^1_{Y_3}$ of vector bundles on $Y_3$,
which is surjective over $U$. Also, we have a natural surjection
$\beta:f^*\Omega^1_X\arrow g^*E$ over $Y_3$ by definition of $E$.
The map $\alpha$ factors through the surjection $\beta $ over $U$,
hence over all of $Y_3$. That is, we have a map $g^*E\arrow \Omega^1_{Y_3}$
of vector bundles over $Y_3$, and it is an isomorphism over $U$.

The resulting
map $H^0(Y_3,S^j(g^*E))\arrow H^0(Y_3,S^j\Omega^1_{Y_3})$ is injective
for all $j>0$. Since $g^*E$ is big on $Y_3$,
$\Omega^1_{Y_3}$ is big.
\end{proof}

\section{Variations of Hodge structure}

Let $X$ be a compact K\"ahler manifold. Consider
a complex variation of Hodge structure $V$ over $X$, and let
$$\varphi:\X\arrow D$$
be the corresponding period map, where $\X$ is the universal cover
of $X$.

\begin{theorem}
Suppose that the derivative of $\varphi$ is injective
at some point of $X$. Then $\Omega^1_X$ is big.
\end{theorem}

The theorem was inspired by Zuo's result
that $\Omega^1_X$ is weakly positive under these assumptions
\cite[Theorem 0.1]{Zuoneg}.
Note that weak positivity (defined in section \ref{curvsect})
generalizes the notion of ``pseudo-effective''
for line bundles. So Zuo's result is similar, but it
does not show that $H^0(X,S^i\Omega^1_X)$ is nonzero for some $i$.
We repeat that our notion of a big vector bundle is not
the stronger notion ``Viehweg big'' (section \ref{curvsect}). 

We recall that a {\it complex variation of Hodge structure }on
a complex manifold $X$ is a complex local system $V$ with an indefinite
hermitian form and an orthogonal $C^{\infty}$ decomposition
$V=\oplus_{p\in \Z}V^p$ such that the form is $(-1)^p$-definite
on $V^p$, and such that Griffiths transversality holds:
the connection sends $V^p$ into
$$A^{1,0}_X(V^{p-1})\oplus A^1_X(V^p)\oplus A^{0,1}_X(V^{p+1})$$
\cite[section 4]{SimpsonHiggs}. Let $r_p=\dim(V^p)$; then the corresponding
period domain is the complex manifold $D=G/V$ where
$G=U(\sum_{p\text{ odd}} r_p,\sum_{p\text{ even}}r_p)$
and $V=\prod_p U(r_p)$. A complex variation of Hodge structure
with ranks $r_p$ is equivalent to a representation
of $\pi_1X$ into $G$ and a $\pi_1X$-equivariant holomorphic
map $\widetilde{X}\arrow D$ which is horizontal with
respect to a natural distribution in the tangent bundle
of $D$.

\begin{proof}
Griffiths and Schmid defined a $G$-invariant hermitian metric
on a period domain $D=G/V$. The period map $\varphi:\X\arrow D$
is always tangent to the ``horizontal'' subbundle of $TD$.
The holomorphic sectional curvatures of $D$
corresponding to horizontal directions are at most a negative
constant \cite[Theorem 9.1]{GriffithsSchmid}. Pulling back the metric
on $D$ gives a canonical hermitian metric $g$ on the Zariski open subset
$U\subset X$ where $\varphi$ is an immersion. Since holomorphic sectional
curvature decreases on submanifolds, $g$ has negative holomorphic
sectional curvature on $U$.
Peters showed that
$g$ has nonpositive holomorphic bisectional
curvature on $U$ \cite[Corollary 1.8, Lemma 3.1]{Peters}. Finally,
$g$ is a K\"ahler metric on $U$ (even though the metric
on $D$ is only a hermitian metric) \cite[Theorem 1.2]{Lu}.

The metric $g$ may degenerate on $X$, but we can
argue as follows. Let $Y\arrow D$ be the Grassmannian bundle
of subspaces of $TD$ of dimension equal to $n=\dim(X)$. Then
the derivative of $\varphi$ gives a lift of the morphism
$\X\arrow D$ to a $\pi_1X$-equivariant meromorphic
map $f:\X\dashrightarrow Y$
(a morphism over $\widetilde{U}$). Let $\X_2$ be the closure
of the graph of $f$ in $\X\times Y$. We have a $\pi_1X$-equivariant
proper bimeromorphic morphism $\X_2\arrow \X$, and $f$ extends
to a morphism $f:\X_2\arrow Y$. Let $X_2=\X_2/\pi_1X$, which
is a compact analytic space with a proper bimeromorphic morphism
$X_2\arrow X$. Finally, let $X_3\arrow X_2$ be a resolution
of singularities; we can assume that $X_3$ is a compact K\"ahler manifold
since $X$ is a compact K\"ahler manifold. Then $X_3$ is a 
compact K\"ahler manifold with a bimeromorphic morphism $X_3\arrow X$,
and $f:\X\dashrightarrow Y$ extends to a $\pi_1X$-equivariant
morphism
$f:\X_3\arrow Y$.

There is a natural $G$-equivariant vector bundle $E$ of rank
$\dim(X)$ on the Grassmannian bundle $Y$, a quotient
of the pullback of $\Omega^1_D$ to $Y$. The bundle $E$
inherits a hermitian metric from $\Omega^1_D$. 
Therefore the bundle $f^*E$ on $\X_3$ has a hermitian metric. This bundle
is $\pi_1X_3$-equivariant, and we also write $f^*E$ for the corresponding
bundle on $X_3$. 
The restriction of $f^*E$ to $U\subset X_3$ can be identified with $\Omega^1_U$
with the metric induced from the metric on the dual bundle $TU$.
Because
curvature increases for quotient bundles \cite[section 7.5]{Zheng},
the curvature of 
$f^*E$ is nonnegative over $U$, hence over all of $X_3$.
Also, by Lemmas \ref{segre} and \ref{curvbound}, the Segr\'e form
$s_n((f^*E)^*)$ is positive at each point of $U$. It may not be positive
on all of $X_3$, but positivity on $U$ implies
that the number $\int_X s_n((f^*E)^*)$ is positive. Equivalently, the
line bundle $O(1)$ on $\P(f^*E)\arrow X_3$ has nonnegative curvature
and the number $(c_1O(1))^{2n-1}$ is positive. So $O(1)$ is nef and big
on $\P(f^*E)$. Equivalently, $f^*E$ is nef and big on $X_3$.

Because we can pull back 1-forms, we have a natural map
$\alpha:\varphi^*\Omega^1_D\arrow \Omega^1_{X_3}$ of vector bundles on $\X_3$,
which is surjective over $\widetilde{U}$. 
Also, we have a natural surjection
$\beta:\varphi^*\Omega^1_D\arrow f^*E$ over $\widetilde{X_3}$
by definition of $E$.
The map $\alpha$ factors through the surjection $\beta $ over $\widetilde{U}$,
hence over all of $\X_3$. That is, we have a map $g^*E\arrow \Omega^1_{Y_3}$
of vector bundles over $\X_3$, and it is an isomorphism over $U$.
This map is $\pi_1X_3$-equivariant, and so we have a corresponding
map of vector bundles over $X_3$.

The resulting
map $H^0(X_3,S^j(f^*E))\arrow H^0(X_3,S^j\Omega^1_{X_3})$ is injective
for all $j\geq 0$. Since $f^*E$ is big on $X_3$,
$\Omega^1_{X_3}$ is big. Since the ring of symmetric differentials on
a compact complex manifold is a bimeromorphic invariant,
$\Omega^1_X$ is big.
\end{proof}

\begin{corollary}
\label{vhs}
Let $X$ be a compact K\"ahler manifold. Consider
a complex variation of Hodge structure $V$ over $X$ with discrete
monodromy group $\Gamma$, and let
$$\varphi:X\arrow D/\Gamma$$
be the corresponding period map. After replacing $X$ by a finite
\'etale covering $Z$, we can assume that $\Gamma$ is torsion-free.
Let $Y$ be a resolution of singularities of the image
of $\varphi:Z\arrow D/\Gamma$. Then $\Omega^1_Y$ is big.
\end{corollary}

For a dominant meromorphic map of compact complex manifolds
$Z\dashrightarrow Y$,
there is a natural pullback map on the ring of symmetric differentials,
$$\oplus_{i\geq 0} H^0(Y,S^i\Omega^1_Y)\arrow \oplus_{i\geq 0}
H^0(Z,S^i\Omega^1_Z),$$
and this is injective. Therefore, Corollary \ref{vhs} implies that
the cotangent dimension $\lambda(X)$ is at least $2\dim(Y)-\dim(X)$,
using that $\lambda(X)=\lambda(Z)$ for a finite \'etale covering
$Z\arrow X$ \cite[Theorem 1]{Sakai}.
In particular,
if $X$ is the base of a variation of Hodge structure
with discrete and infinite monodromy, then the image $Y$
of the period map has positive dimension, and so our lower bound
for $\lambda(X)$ gives that $X$
has a nonzero symmetric differential.

\section{Non-rigid representations}

We use the following result which Arapura proved for smooth
complex projective varieties \cite[Proposition 2.4]{Arapura}. It was
extended to compact K\"ahler manifolds by the second author
\cite[Theorem 1.6(i)]{Klingler}.

\begin{theorem}
\label{nonrigid}
Let $X$ be a compact K\"ahler manifold. Suppose
that $\pi_1X$ has a complex
representation of dimension $n$ which is not rigid.
Then $H^0(X,S^i\Omega^1_X)\neq 0$ for some $1\leq i\leq n$.
\end{theorem}

Here we say that a representation of $\pi_1X$ is rigid if
the corresponding point in the moduli space $M_B(X,GL(n))$
(the ``Betti moduli space'' or ``character variety'')
is isolated. The points of $M_B(X,GL(n))$ are in one-to-one
correspondence with the isomorphism classes of $n$-dimensional
semisimple representations (meaning direct sums of irreducibles)
of $\pi_1X$, with a representation
being sent to its semisimplification \cite[section 7]{SimpsonII}.

We also need the following $p$-adic analogue,
essentially proved by Katzarkov and Zuo
using Gromov-Schoen's construction
of pluriharmonic maps into the Bruhat-Tits building
\cite[proof of Theorem 3.2]{Katzarkov},
\cite[section 1]{ZuoCrelle}, \cite{GromovSchoen}.
An explicit formulation and proof of Theorem \ref{unbounded}
can be found in the second author's \cite[Theorem 1.6(ii)]{Klingler}.

\begin{theorem}
\label{unbounded}
Let $X$ be a compact K\"ahler manifold, and let $K$ be a nonarchimedean
local field. Suppose
that there is a semisimple representation from $\pi_1X$
to $GL(n,K)$ which is unbounded
(equivalently, which is not conjugate to a representation
over the ring of integers of $K$).
Then $H^0(X,S^i\Omega^1_X)\neq 0$ for some $1\leq i\leq n$.
\end{theorem}

\section{Representations in positive characteristic}

We now prove Theorem \ref{main} for representations in positive
characteristic, which turns out to be easier. That is, we will
show that if $X$ is a compact K\"ahler manifold such that
$\pi_1X$ has a finite-dimensional infinite-image
representation over some field $k$
of characteristic $p>0$, then
$X$ has a nonzero symmetric differential.

We can assume that the field $k$ is algebraically closed.
By Procesi, for each natural number $n$,
there is an affine scheme $M=M_B(X,GL(n))_{\F_p}$
of finite type over $\F_p$
whose $k$-points are in one-to-one correspondence
with the set of isomorphism classes of semisimple 
$n$-dimensional representations of $\pi_1X$ over $k$
\cite[Theorem 4.1]{Procesi}.
(To construct this scheme, choose a finite presentation for $\pi_1X$.
The space of homomorphisms $\pi_1X\arrow GL(n)$ is a closed subscheme
of $GL(n)^r$ over $\F_p$
in a natural way, where $r$ is the number of generators
for $\pi_1X$. We then take the affine GIT quotient by the conjugation
action of $GL(n)$.)

We want to show that if $X$ has no symmetric differentials, then
every $n$-dimensional representation of $\pi_1X$ over $k$ has finite image.
Suppose that the Betti moduli space $M$
has positive dimension over $\F_p$. Then $M$ contains an affine curve
over $\F_q$ for some power $q$ of $p$. Therefore, after possibly increasing
$q$, there is a point of $M(\F_q((t)))$ which is not
in $M(\F_q[[t]])$. After increasing $q$ again, it follows that
there is a semisimple
representation of $\pi_1(X)$ over $\F_q((t))$ which is not defined
over $\F_q[[t]]$. This contradicts Katzarkov and Zuo's Theorem \ref{unbounded},
since $X$ has no symmetric differentials.
So in fact $M$ has dimension zero over $\F_p$.

It follows that every finite-dimensional
semisimple representation of $\pi_1X$ over $k$
is in fact defined over $\FF$. But every finite-dimensional
representation of a finitely
generated group over $\FF$ has finite image, since the generators all
map to matrices over some finite field. So every finite-dimensional
semisimple representation
of $\pi_1X$ over $k$ has finite image.

Finally, we show that any finite-dimensional representation $\rho$ of $\pi_1X$
over $k$ has finite image. We know that the semisimplification of $\rho$
has finite image. Therefore, a finite-image subgroup $H$ of $\pi_1X$
maps into the subgroup of strictly upper-triangular matrices in $GL(n,k)$.
The latter group is a finite extension of copies of the additive group
over $k$, and so the image of $H$ is a finite extension of abelian groups
killed by $p$. Since $H$ is finitely generated, it follows that the image
of $H$ is finite. Therefore the image of $\pi_1X$ in $GL(n,k)$
is finite, as we want.

\section{Representations in characteristic zero}

\begin{theorem}
\label{char0}
Let $X$ be a compact K\"ahler manifold. Suppose that
there is a finite-dimensional complex representation of $\pi_1X$
with infinite image. Then $X$ has a nonzero symmetric differential.
\end{theorem}

If a finitely generated
group has an infinite-image representation
over a field of characteristic zero, then it has an infinite-image
representation over $\C$. So Theorem \ref{char0} will complete the proof
of Theorem \ref{main}.

\begin{proof}
By Theorem \ref{nonrigid}, we can assume that the given representation $\rho$
of $\pi_1X$ is rigid. Therefore, the point associated to $\rho$
in $M_B(X,GL(n))$ is fixed by Simpson's $\C^*$ action, which means
that the semisimplification $\sigma$ of $\rho$ can be made
into a complex variation of Hodge structure over $X$
\cite[Corollary 4.2]{SimpsonHiggs}.

Suppose that the representation $\sigma$ has finite image. Then $\rho$ sends
a finite-index subgroup $H$ of $\pi_1X$ into the group $U$ of
strictly upper-triangular matrices in $GL(n,\C)$. 
Since $U$ is nilpotent and $\rho$ has infinite image,
the abelianization of $H$ must be infinite. By Hodge theory,
the finite \'etale covering
$Y\arrow X$ corresponding to $H$ has a nonzero 1-form
$\alpha \in H^0(Y,\Omega^1_Y)$.
If $H$ has index $r$ in $\pi_1X$, the norm of $Y$ is a nonzero element of
$H^0(X,S^r\Omega^1_X)$, as we want.

It remains to consider the case where $\sigma$ is a complex
variation of Hodge structure with infinite image. We cannot
immediately apply Corollary \ref{vhs} because the image of $\sigma$
in $GL(n,\C)$ need not be discrete. At least $\sigma$ is conjugate
to a representation into $GL(n,F)$ for some number field,
because $\sigma $ is rigid. More precisely, $\sigma$ is a complex
direct factor of a $\Q$-variation of Hodge structure $\tau:\pi_1X
\arrow GL(m,\Q)$
\cite[Theorem 5]{SimpsonHiggs}. Let $m$ be the dimension of $\tau$.

For each prime number $p$, consider the representation
$\tau:\pi_1X\arrow GL(m,\Q_p)$. By Katzarkov and Zuo's Theorem \ref{unbounded},
if this representation is not bounded,
then $H^0(X,S^i\Omega^1_X)$ is nonzero for
some $1\leq i\leq m$.

Therefore, we can assume that $\tau$ is $p$-adically bounded for each
prime number $p$. Then $\tau$ is conjugate to a representation
into $GL(m,\Z)$ \cite{Bass}. Thus $\tau$ is a complex variation
of Hodge structure with discrete monodromy group.
By Corollary \ref{vhs}, $X$ has nonzero symmetric
differentials. More precisely, there is a finite \'etale
covering $Z$ of $X$ and a blow-up $Z_2$ of $Z$
such that the given representation of $\pi_1Z_2=\pi_1Z\subset \pi_1X$ factors
through a surjection $Z_2\arrow Y$ with $\Omega^1_Y$ big and $Y$ 
of positive dimension. As a result,
the cotangent dimension $\lambda(X)$ is at least $2\dim(Y)-\dim(X)$.
\end{proof}


\small \sc Institut de Math\'ematiques de Jussieu, Paris, France

brunebarbe@math.jussieu.fr
\medskip

Institut de Math\'ematiques de Jussieu, Paris, France

klingler@math.jussieu.fr
\medskip

DPMMS, Wilberforce Road,
Cambridge CB3 0WB, England

b.totaro@dpmms.cam.ac.uk
\end{document}